\documentclass[12pt,oneside]{amsart}
\usepackage{amssymb}
\usepackage{amsmath}
\numberwithin{equation}{section}
\begin{document}


\newcommand{\pntwo}{\mathbb{P}^{n_2} \times \cdots \times \mathbb{P}^{n_k}}
\newcommand{\bj}{\underline{j}}
\newcommand{\bi}{\underline{i}}
\newcommand{\dep}{\operatorname{depth}}
\newcommand{\hit}{\operatorname{ht}}
\newcommand{\opdd}{\overline{P}_{d_1,d_2}}
\newcommand{\opdk}{\overline{P}_{d_1,\ldots,d_k}}
\newcommand{\op}{\overline{P}}
\newcommand{\xab}{X^{\underline{a}}Y^{\underline{b}}}
\newcommand{\xka}{X_1^{\underline{\alpha}_1}\cdots X_k^{\underline{\alpha}_k}}
\newcommand{\bb}{{\bf b}}
\newcommand{\ab}{(\underline{a},\underline{b})}
\newcommand{\alphak}{(\underline{\alpha}_1,\ldots,\underline{\alpha}_k)}
\newcommand{\abo}{(\overline{\underline{a},\underline{b}})}
\newcommand{\alphako}{\overline{(\underline{\alpha}_1,\ldots,\underline{\alpha}_k)}}
\newcommand{\opn}{\mathcal{O}_{\pr^n}}
\newcommand{\opthree}{\mathcal{O}_{\pr^3}}
\newcommand{\otfold}{\mathcal{O}_{\mathbb{T}}}
\newcommand{\Iz}{I_{\Z}}
\newcommand{\Ixp}{I_{\xp}}
\newcommand{\Z}{\mathbb{Z}}
\newcommand{\xpi}{\X_{P_i}}
\newcommand{\xq}{\X_{Q_1}}
\newcommand{\xqi}{\X_{Q_i}}
\newcommand{\xp}{\X_{P_1}}
\newcommand{\lp}{L_{P_1}}
\newcommand{\ax}{\alpha_{\X}}
\newcommand{\bx}{\beta_{\X}}
\newcommand{\kxo}{k[x_0,\ldots,x_n]}
\newcommand{\kx}{k[x_1,\ldots,x_n]}
\newcommand{\popo}{\mathbb{P}^1 \times \mathbb{P}^1}
\newcommand{\prthree}{\mathbb{P}^1 \times \mathbb{P}^1 \times \mathbb{P}^1}
\newcommand{\pr}{\mathbb{P}}
\newcommand{\pn}{\mathbb{P}^n}
\newcommand{\pnpm}{\mathbb{P}^n \times \mathbb{P}^m}
\newcommand{\pnk}{\mathbb{P}^{n_1} \times \cdots \times \mathbb{P}^{n_k}}
\newcommand{\X}{\mathbb{X}}
\newcommand{\Y}{\mathbb{Y}}
\newcommand{\N}{\mathbb{N}}
\newcommand{\M}{\mathbb{M}}
\newcommand{\Q}{\mathbb{Q}}
\newcommand{\tfold}{\mathbb{T}}
\newcommand{\Ix}{I_{\X}}
\newcommand{\pix}{\pi_1(\X)}
\newcommand{\pixt}{\pi_2(\X)}
\newcommand{\pipi}{\pi_1^{-1}(P_i)}
\newcommand{\piqi}{\pi_2^{-1}(Q_i)}
\newcommand{\qpi}{Q_{P_i}}
\newcommand{\pqi}{P_{Q_i}}
\newcommand{\pitk}{\pi_{2,\ldots,k}}
\newcommand{\Ssx}{\mathcal S_{\X}}
\newcommand{\ui}{\underline{i}}
\newcommand{\uj}{\underline{j}}
\newcommand{\ua}{\underline{\alpha}}
\newcommand{\pnr}{\mathbb{P}^{n_1}\times \cdots \times \mathbb{P}^{n_r}}
\newcommand{\uk}{\underline{k}}

\newtheorem{theorem}{Theorem}[section]
\newtheorem{corollary}[theorem]{Corollary}
\newtheorem{proposition}[theorem]{Proposition}
\newtheorem{lemma}[theorem]{Lemma}
\newtheorem{question}[theorem]{Question}
\newtheorem{problem}{Problem}
\newtheorem{conjecture}[theorem]{Conjecture}

\newenvironment{remark}
{\vspace{.15cm} \refstepcounter{theorem} \noindent{\bf Remark
\thetheorem.}} {\vspace{.15cm} }
\newenvironment{example}
{\vspace{.15cm} \refstepcounter{theorem} \noindent{\bf Example
\thetheorem.}} {\vspace{.15cm} }
\newenvironment{definition}
{\vspace{.15cm} \refstepcounter{theorem} \noindent{\bf Definition
\thetheorem.}} {\vspace{.15cm} }

\newcommand{\m}[1]{\marginpar{\addtolength{\baselineskip}{-3pt}{\footnotesize
\it #1}}}


\title{On the Hilbert functions of sets of points in $\prthree$}

\author{Elena Guardo}
\address{Dipartimento di Matematica e Informatica\\
Viale A. Doria, 6 - 95100 - Catania, Italy}
\email{guardo@dmi.unict.it}
\urladdr{http://www.dmi.unict.it/$\sim$guardo/}

\author{Adam Van Tuyl}
\address{Department of Mathematical Sciences \\
Lakehead University \\
Thunder Bay, ON P7B 5E1, Canada}
\email{avantuyl@lakeheadu.ca}
\urladdr{http://flash.lakeheadu.ca/$\sim$avantuyl/}


\begin{abstract}
Let $H_X$ be the trigraded Hilbert function of a set $X$ of reduced points
in $\prthree$.  We show how to extract some geometric information about
$X$ from $H_X$.  This note generalizes a similar result of Giuffrida,
Maggioni, and Ragusa about sets of points in $\popo$.
\end{abstract}
\maketitle


\section{Introduction}

Let $K$ be an algebraically closed field with char$(K)=0$, and
suppose that $X$ is a finite set of points in the projective space
$\mathbb{P}^n$ over $K$.  The {\it Hilbert function} of $X$ is the
numerical function $H_X:\mathbb{N} \rightarrow \mathbb{N}$ defined
by $H_X(i) = \dim_K (R/I(X))_i$ where $R = K[x_0,\ldots,x_n]$ and
$I(X)$ is the homogeneous ideal associated to $X$. The Hilbert
functions of sets of points are a well studied object, see for
example \cite{GGR,GHS,GeMR,MR}. Among the questions one can ask is
the question of what geometric information about the set of
points, e.g., the number of collinear points in $X$, can be
inferred from the function $H_X$.  For work on problems of this
type, we point the readers to \cite{BGM,CM}.

Giuffrida, Maggioni, and Ragusa \cite{GuMaRa} were among the
first to consider the Hilbert function of a set of points $X$
in a multiprojective
space $\mathbb{P}^{n_1} \times \cdots \times \mathbb{P}^{n_r}$.
In this context, the associated ideal $I(X)$ is a multihomogeneous
ideal in an $\mathbb{N}^r$-graded polynomial ring $R$.  The Hilbert
function of $X$ is the function $H_X:\mathbb{N}^r \rightarrow \mathbb{N}$
defined by $H_X(\underline{i}) = \dim_K (R/I(X))_{\underline{i}}$.  Unlike the
singly graded case, our understanding of these  functions is far from complete.
Most notably, while there exists a classification for the Hilbert
functions of sets of points in $\mathbb{P}^n$ (see \cite{GeMR}),
no classification is known in the multigraded situation, including
$\popo$.  Some known results can be found in \cite{Gu2,GuMaRa,VT1}.

In this note, we introduce some new results about the Hilbert
functions of points $X$ in $\prthree$ (which can be scaled to
$\mathbb{P}^1 \times \cdots \times \mathbb{P}^1$; see Remark \ref{scale}).  
In particular,
in the spirit of \cite{BGM,CM}, we describe how geometric information
about $X$ is encoded into the Hilbert function $H_X$.  Specifically,
Theorem \ref{genGMR} shows
how the number of points on a ``line'' (see Definition
\ref{line}) in $\prthree$ is captured.
We were inspired by a similar result of Giuffrida, Maggioni,
and Ragusa \cite[Theorem 2.12]{GuMaRa} for points in $\popo$.  Our proof,
however, uses a different approach.


\section{Preliminaries}

We introduce the necessary notation and basic results concerning
sets of points in $\prthree$. Throughout, we use $\succeq$ to
denote the natural partial order on $\mathbb{N}^3$ defined by
$(i,j,k) \succeq (i',j',k')$ if and only if  $i\geq i'$, $j \geq
j'$, and $k \geq k'$. We set $R = K[x_0,x_1,y_0,y_1,z_0,z_1]$ and
induce a trigrading by setting $\deg x_i = (1,0,0)$, $\deg y_i =
(0,1,0)$ and $\deg z_i = (0,0,1)$ for $i=1,2$.

Suppose that
\[
P = [a_0:a_1] \times [b_0:b_1] \times [c_0:c_1]\in \prthree
\]
is a point in $\prthree$.  Associated to $P$ is
the trihomogeneous ideal given by
\[
I(P) = (a_1x_0 - a_0x_1,b_1y_0 - b_0y_1,c_1z_0 - c_0z_1).
\]
Given any finite set of distinct points $X = \{P_1,\ldots,P_s\}$
in $\prthree$, its associated ideal is the trihomogeneous
ideal $I(X) = \bigcap_{i=1}^s I(P_i)$.

We will sometimes write $P$ as $P = A \times B \times C$, with
$A,B,C \in \pr^1$, i.e., we will use $A$'s for the first coordinate,
and so on.  We shall sometimes write the generators
of $I(P)$ as $I(P) = (L_A,L_B,L_C)$ where $L_{A}$ is the
form of degree $(1,0,0)$, $L_B$ is the form of degree $(0,1,0)$,
and $L_C$ is the form of degree $(0,0,1)$.

For each $i=1,2,3$, let $\pi_i:\prthree \rightarrow \pr^1$ denote
the natural projection map.  Consequently, $\pi_1(X) = \{A_1,\ldots,A_{t_1}\}$,
$\pi_2(X) = \{B_1,\ldots,B_{t_2}\}$, and $\pi_3(X) = \{C_1,\ldots,C_{t_3}\}$
denote the sets of distinct first, second, and third coordinates of $X$
respectively.   We let $t_1 = |\pi_1(X)|, t_2 = |\pi_2(X)|$,
and $t_3 = |\pi_3(X)|$.

\begin{definition}
The {\it Hilbert function of $X$}, denoted $H_X$, is the
function $H_X:\mathbb{N}^3 \rightarrow \mathbb{N}$ defined by
$H_{X}(i,j,k) = \dim_K R_{i,j,k} - \dim_K I(X)_{i,j,k}.$
\end{definition}

Note that $\dim_K R_{i,j,k} = (i+1)(j+1)(k+1)$
since there are $i+1$ ways to make a monomial of degree $i$ in the
variables $x_0$ and $x_1$, $j+1$ ways to make a monomial of degree
$j$ in the $y_i$s, and $k+1$ ways to make
a monomial of degree $k$ in the $z_i$s.

Our goal is to show how $H_X$ captures
geometric information about the number of points lying
on linear subvarieties of $\prthree$.  We give a more precise
definition.

\begin{definition}\label{line}
Let $L \in R_{1,0,0}$ and $L' \in R_{0,1,0}$.  We call the variety
$\mathcal{L}$
in $\prthree$
defined by the ideal $(L,L') \subseteq R$ a {\it line of type $(1,1,0)$}.
\end{definition}

\begin{remark}
We will focus only on lines of type $(1,1,0)$, although similar results
can be proved for lines of type $(1,0,1)$ and $(0,1,1)$, which
are defined in an analogous manner.
\end{remark}

\begin{remark}
We add a few comments about
how to interpret the geometry.
One can construct an embedding of $\prthree$
into a projective space; in particular, using Segre's embedding
$\tfold = \prthree \hookrightarrow \pr^7$
using the sheaf $\otfold(1,1,1)\,$.  We can easily check that the
ideal of the image $\tfold'$ of
$\tfold$ under the embedding is:
\small
\begin{eqnarray*}
I(\tfold') & = & (u_0u_7-u_1u_6,u_0u_7-u_2u_5,u_0u_7-u_3u_4, u_0u_3-u_1u_2,
u_4u_7-u_5u_6, \\
&&u_0u_5-u_1u_4, u_2u_7-u_3u_6,u_0u_6-u_2u_4,u_1u_7-u_3u_5)
\end{eqnarray*}
\normalsize
in $K[u_0,\ldots,u_7]$.    We can view $\prthree$ as ruled
$3$-fold.  The lines of type $(1,1,0)$, respectively $(1,0,1)$ and $(0,1,1)$,
can be viewed as rulings on this surface.  Our goal is to count the number
of points on these rulings using the Hilbert function.

Alternatively, if we fix a degree $(1,0,0)$ form $L$,
we can view this as fixing a divisor of type $(1,0,0)$.  Fixing
this linear equation is equivalent to fixing a point $A$ in the first
copy of $\pr^1$.  We are then looking at points in the set
\[S_A = \{A \times B \times C \in \popo \times \mathbb{P}^1 ~|~ B \times 
C \in \popo \} \cong \popo\]
which is also isomorphic to the ruled quadric surface in $\mathbb{P}^3$.
So,
when we form a line of type $(1,1,0)$, we can view it as fixing
a point in a $\mathbb{P}^1$ and a line in a $\popo$, or equivalently, 
a ruling on the ruled quadric surface $Q$, 
and then we are counting the number of points
on this ruling on the quadric $Q$.  Note, by abuse of notation,
the divisor of type $(1,0,0)$ can be called a plane of type $(1,0,0)$,
and similarly for divisor of type $(0,1,0)$.  So, the intersection
of these two planes results in a line of type $(1,1,0)$.
\end{remark}

We now provide a number of lemmas that shall be required for our main
result.

\begin{lemma}\label{stabilize}
Let $X$ be a finite set of distinct points in $\prthree$.
Fix an integer $k \geq 0$.  Then
\[H_X(t_1-1,t_2-1,k) = H_X(i,j,k)
~~\mbox{for all $(i,j,k) \succeq (t_1-1,t_2-1,k)$}.\]
\end{lemma}

\begin{proof}
Let $L$ be a form of degree $(1,0,0)$ that such that no point
of $X$ lies on $L$.  Let $L'$ be a form of degree $(0,1,0)$ such
that no point of $X$ lies on $L'$.  We observe that
$L$, respectively $L'$, is a nonzerodivsor on $R/I(X)$.
Thus the multiplication map
\[\times \overline{L}: (R/I(X))_{i,j,k} \rightarrow (R/I(X))_{i+1,j,k}\]
is injective for all $(i,j,k)$.  A similar result
holds using $L'$.  Thus
\[H_X(t_1-1,t_2-1,k) \leq H_X(t_1,t_2-1,k)\leq \cdots \leq H_X(i,t_2-1,k)
~~\mbox{for $i \geq t_1-1$}\]
and
\[H_X(i,t_2-1,k) \leq
H_X(i,t_2,k) \leq \cdots \leq H_X(i,j,k)
~~\mbox{for $j \geq t_2-1$.}\]

In addition, we have the short exact sequence
\[0 \longrightarrow R/I(X)(-1,0,0)
\stackrel{\times \overline{L}}{\longrightarrow} R/I(X) \longrightarrow
R/(I(X),L) \longrightarrow 0.\]
Now $H_X(t_1-1,0,0) = H_X(t_1,0,0) = |\pi_1(X)|$ because
$\bigoplus_{i \in \mathbb{N}}I(X)_{i,0,0} \cong I(\pi_1(X)) \subseteq K[x_0,x_1]$,
i.e., the
ideal of the $t_1$ points $\pi_1(X)$ in $\pr^1$.  So, the short
exact sequence implies that $(I(X),L)_{i,0,0} = R_{i,0,0}$ for all
$i \geq t_1$.  But this means that the multiplication map $\times \overline{L}$
is also surjective for $(i,j,k)$ with $i \geq t_1-1$.  Thus
$H_X(t_1-1,t_2-1,k) = H_X(t_1,t_2-1,k) = \cdots
= H_X(i,t_2-1,k)$.

We now apply a similar argument with $L'$ to show that
$H_X(i,t_2-1,k) = H_X(i,t_2,k) = \cdots = H_X(i,j,k)$.
\end{proof}

\begin{lemma}\label{pointsonline}
Let $\mathcal{L}$ be a line of type $(1,1,0)$, and suppose that
$X$ is a finite set of distinct points in $\prthree$ such that
$X \subseteq \mathcal{L}$.  If $|X| = s$, then
\[H_{X}(i,j,k) = \min\{k+1,s\} ~~\mbox{for all $(i,j,k) \in \mathbb{N}^3$}.\]
\end{lemma}

\begin{proof}
Because $X \subseteq \mathcal{L}$, the set $X$ has the form
\[X = \{A \times B \times C_1,A \times B \times C_2, \ldots, A \times B \times
C_s \}\]
for some $A,B, C_i \in \mathbb{P}^1$.  After a change of coordinates,
we can assume that $A = [1:0]$ and $B = [1:0]$.  Thus
\[I(X) = \bigcap_{i=1}^s (x_1,y_1,L_{C_i}) = (x_1,y_1,L_{C_1}L_{C_2}\cdots L_{C_s}).\]
If we set $L = L_{C_1}\cdots L_{C_s}$, then
$R/I(X) \cong K[x_0,y_0,z_0,z_1]/(L)$;  here, we are
viewing $L$ as an element of $K[x_0,y_0,z_0,z_1]$.
If $S = K[x_0,y_0,z_0,z_1]$ is the $\mathbb{N}^3$-graded ring
with $\deg x_0 = (1,0,0)$, $\deg y_0 = (0,1,0)$ and $\deg z_i = (0,0,1)$,
then the result follows by using the short exact sequence
\[0 \longrightarrow S(0,0,-s) \stackrel{\times L}{\longrightarrow}
S \longrightarrow S/(L) \longrightarrow 0\]
to compute $H_X$ and the fact that
$\dim_K S_{i,j,k} = k+1$ for all $(i,j) \in \mathbb{N}^2$.
\end{proof}

\begin{lemma}\label{zerohilbertfunction}
Let $X$ be a finite set of distinct points in $\prthree$, and suppose that
$\mathcal{L}$ is a line of type $(1,1,0)$ that intersects
$X$, but $X \not\subseteq \mathcal{L}$.
Set $X_2 = X \cap \mathcal{L}$ and $X_1 = X \setminus X_2$.
Then
\[R_{t_1-1,t_2-1,0} = (I(X_1) + I(X_2))_{t_1-1,t_2-1,0}.\]
Consequently,
\[H_{R/(I(X_1)+I(X_2))}(t_1-1,t_2-1,k) = 0 ~~\mbox{for all $k \geq 0$}.\]
\end{lemma}

\begin{proof}
The second statement follows from the first
since $(I(X_1)+I(X_2))_{t_1-1,t_2-1,k} = R_{t_1-1,t_2-1,k}$ for all
$k \geq 0$ if the first statement holds.

After a change of coordinates, we can assume that $\mathcal{L}$ is the
line defined by $x_1$ and $y_1$, i.e., every point on
$X_2$ has the form $[1:0] \times [1:0] \times C$ for some $C \in \mathbb{P}^1$.
As in the proof of Lemma \ref{pointsonline}, we can take
$I(X_2) = (x_1,y_1,L)$ where $L$ is some homogeneous polynomial only in
the variables $z_0$ and $z_1$.  Now for any $(i,j) \in \mathbb{N}^2$,
$I(X_2)_{i,j,0}$  contains all the monomials of degree $(i,j,0)$
except $x_0^iy_0^j$.

Now $\pi_1(X) = \{A_1,\ldots,A_{t_1-1},[1:0]\}$ and $\pi_2(X) = \{B_1,
\ldots,B_{t_2-1},[1:0]\}$.  Note that because $X \not\subseteq \mathcal{L}$,
either $t_1 \geq 2$ or $t_2 \geq 2$.
Let $L_{A_i}$ be the form of degree $(1,0,0)$
that vanishes at points of the form $A_i \times B \times C$
where $B,C \in \mathbb{P}^1$
and let $L_{B_j}$ be the form of degree $(0,1,0)$
that vanishes at all points of the form $A \times B_j \times C$
with $A,C \in \mathbb{P}^1$.
Then $H = L_{A_1}\cdots L_{A_{t_1-1}}L_{B_1}\ldots,L_{B_{t_2-1}}$ is a degree
$(t_1-1,t_2-1,0)$ 
form that vanishes at all the points of $X_1$ with $t_1-1>0$ or $t_2-1>0$, 
that is, $(t_1-1,t_2-1,0) \neq (0,0,0)$.
Thus $H \in I(X_1)_{t_1-1,t_2-1,0}$ and $H$ is not a constant.  
Furthermore, since $H$ does
not vanish at any of the points in $X_2$, neither $x_1$ nor
$y_1$ divides $H$.  So $H$ must have
the form $H = cx_0^{t_1-1}y_0^{t_2-1} + H'$ where $c\neq 0$.

Since $H \in I(X_1)_{t_1-1,t_2-1,0}$ and since all the monomials of degree
$(t_1-1,t_2-1,0)$ except $x_0^{t_1-1}y_0^{t_2-1}$ belong to $I(X_2)_{t_1-1,t_2-1,0}$,
we must have all the monomials of degree $(t_1-1,t_2-1,0)$ belonging
to $(I(X_1)+I(X_2))_{t_1-1,t_2-1,0}$, from which our conclusion follows.
\end{proof}


\section{Main Result}

In this section we prove our main result which explains how to
extract geometric information about $X$ from $H_X$.  In particular,
we are able to determine information about the number of points
on lines of type $(1,1,0)$.

We make use of the following notation.   If $X$ is a finite
set of distinct points in $\prthree$, for each integer $i \geq 1$, let
\[r_i(X) := \mbox{the number of lines of type $(1,1,0)$ that contain
exactly $i$ points of $X$}.\]
Because $X$ is finite, $r_i(X) = 0$ for all but a finite number
of $i$.   Given $H_X$, we set
\[d_{i,j,k} : = H_{X}(i,j,k) - H_{X}(i,j,k-1) ~~\mbox{for all $(i,j,k) \in
\mathbb{N}^3$}\]
where $H_X(a,b,c) =0$ if $(a,b,c) \not\succeq (0,0,0)$.
We shall be interested in the sequence
$\{d_{t_1-1,t_2-1,k} - d_{t_1-1,t_2-1,k+1}\}_{k \in \mathbb{N}}$.
It is not obvious that this is a sequence of nonnegative numbers.
However, even though one can prove this directly from the properties
of the Hilbert function,  this fact will be an immediate consequence
of our main theorem, stated below, which shows how geometric information
is encoded into the sequence.

\begin{theorem} \label{genGMR}
Let $X$ be a finite set of distinct points in $\prthree$.
Let $t_1 = |\pi_1(X)|$ and $t_2 = |\pi_2(X)|$.
For every $k \geq 0$,
\begin{eqnarray*}
r_{k+1}(X) &=& d_{t_{1}-1,t_{2}-1,k}-d_{t_{1}-1,t_{2}-1,k+1}\\
&=& 2H_{X}({t_1-1,t_2-1,k}) - H_{X}({t_1-1,t_2-1,k-1})- H_{X}({t_1-1,t_2-1,k+1}).
\end{eqnarray*}
\end{theorem}

\begin{remark}
As we see in the above theorem, $r_{k+1}(X)$ can be computed
directly from the Hilbert function $H_X$ or through the
$d_{i,j,k}$s.  We have presented the formula for $r_{k+1}(X)$ using 
the $d_{i,j,k}$s because it is reminiscent of the formulas found in
\cite[Theorem 2.12]{GuMaRa} for points in $\popo$ which inspired this
result.
\end{remark}

Let us
illustrate this result before turning to its proof.

\begin{example}
We show how to go from the Hilbert function to information about
the set of points. Suppose that we are given the following
trigraded Hilbert function for a set of points $X$. Below, $H_X$
is written as a collection of infinite matrices, where the initial
row and column are indexed with $0$ as opposed to $1$.  You should
view these matrices as the ``layers'' of the box matrix that make
up $H_X$.
\footnotesize
\begin{eqnarray*}
H_{X}(i,j,0)&=&
\begin{bmatrix}
1&$2$&$3$&$3$&\cdots\\
2&$4$&$5$&$5$&\cdots\\
3&$5$&$6$&$6$&\cdots\\
3&$5$&$6$&$6$&\cdots\\
\vdots&\vdots&\vdots&\vdots&\ddots\\
\end{bmatrix} \quad
H_{X}(i,j,1) =
\begin{bmatrix}
2&4&5&5&\cdots\\
4&7&8&8&\cdots\\
5&8&9&9&\cdots\\
5&8&9&9&\cdots\\
\vdots&\vdots&\vdots&\vdots&\ddots\\
\end{bmatrix} \quad
H_{X}(i,j,2)  =
\begin{bmatrix}
3&6&7&7&\cdots\\
5&9&10&10&\cdots\\
6&10&11&11&\cdots\\
6&10&11&11&\cdots\\
\vdots&\vdots&\vdots&\vdots&\ddots\\
\end{bmatrix}
\end{eqnarray*}
\normalsize 
and $H_X(i,j,k) = H_X(i,j,2)$ for all $k \geq 2$.

From this Hilbert function, we can see that $t_1=t_2 = t_3 = 3$.
We can use Theorem \ref{genGMR} to
determine the number of points on each line of type $(1,1,0)$.
We need the values
of $H_X(2,2,l)$ for all $l \in \mathbb{N}$.  From $H_X$ above, we have
\[H_X(2,2,0) = 6,~~ H_X(2,2,1) = 9, ~~\mbox{and}~~ H_X(2,2,l) = 11
~~\mbox{for all $l \geq 2$}.\]
So, $d_{2,2,0} = 6,$ $d_{2,2,1} = 3,$ $d_{2,2,2} = 2$, and $d_{2,2,l} = 0$
otherwise.
By Theorem \ref{genGMR}, we then
have $d_{2,2,0} - d_{2,2,1} = 3$ lines
of type $(1,1,0)$ that contain exactly one point of $X$,
$d_{2,2,1}-d_{2,2,2} = 1$ lines of type $(1,1,0)$ that contain
exactly two points of $X$, and $d_{2,2,2}-d_{2,2,3} = 2$ lines
of type $(1,1,0)$ that contain exactly three points of $X$.

Indeed, $H_X$ is the Hilbert function of the set of points
$X$ constructed as below.
Let $A_i := [1:i] \in \pr^1,$ $B_i := [1:i] \in \pr^1$ and $C_i :=
[1:i] \in \pr^1$.
Then $X$ is the following scheme
\[X:=\{P_{111},P_{112},P_{113},P_{121},P_{122},P_{123},P_{212},
P_{211},P_{311},P_{221},P_{131}\}\] \noindent where
$P_{ijk}=A_i\times B_j\times C_k$.
If $L_{A_i}$ denotes the $(1,0,0)$
form that vanishes at all points of $\prthree$ whose first
coordinate is $A_i$, and similarly, if $L_{B_j}$ denotes the $(0,1,0)$
form that vanishes at all points whose second coordinate is $B_j$,
then $(L_{A_1},L_{B_1})$ and $(L_{A_1},L_{B_2})$ are the two lines
of type $(1,1,0)$ that contain exactly three points of $X$.
\end{example}

Theorem \ref{genGMR} will follow from the next result.

\begin{theorem}\label{sum}
Let $X$ be a finite set of distinct points in $\prthree$.
Let $t_1 = |\pi_1(X)|$ and $t_2 = |\pi_2(X)|$.  Then for
all $k \geq 0$,
\[H_X(t_1-1,t_2-1,k) = \sum_{m=1}^{k+1}\left(\sum_{n=k+1}^\infty r_n(X)\right).\]
\end{theorem}

\begin{proof}
We do induction on the number of lines of type $(1,1,0)$ that
intersect with $X$.  For the base case, suppose that there is only
one line $\mathcal{L}$ of type $(1,1,0)$ that intersects $X$. In
other words, $X$ is contained on $\mathcal{L}$, whence $t_1 = t_2
=1$.  If $|X| = s$, then $r_s(X) = 1$ and $r_{s'}(X) = 0$ for all
$s' \neq s$. By Lemma \ref{pointsonline},
\[H_X(0,0,k) = \min\{k+1,s\} ~~\mbox{for all $k \geq 0$}.\]
On the other hand, note that
$\sum_{n=k+1}^\infty r_n(X) = 1$  if $k+1 \leq s$ and zero otherwise.
This implies that
\[\sum_{m=1}^{k+1}\left(\sum_{n=k+1}^\infty r_n(X)\right) = \min\{k+1,s\}.\]
So the formulas agree.

For the induction step, suppose that there are
$t >  1$ lines $\mathcal{L}_1,\ldots,
\mathcal{L}_t$ of type $(1,1,0)$ that intersect $X$.
Set $X_2 = X \cap \mathcal{L}_t$ and $X_1 = X \setminus X_2$.
We now consider the following short exact sequence
\[0 \longrightarrow R/(I(X_1) \cap I(X_2)) \longrightarrow
R/I(X_1) \oplus R/I(X_2) \longrightarrow R/(I(X_1)+I(X_2)) \longrightarrow
0.\]
From this short exact sequence we have
\footnotesize
\begin{equation}\label{ses}
H_X(t_1-1,t_2-1,k) = H_{X_1}(t_1-1,t_2-1,k) + H_{X_2}(t_1-1,t_2-1,k)
- H_{R/(I(X_1)+I(X_2))}(t_1-1,t_2-1,k)
\end{equation}
\normalsize
for all $k \geq 0$.

Since there are only $t-1$ lines of type $(1,1,0)$ that intersect
with $X_1$ and only one line of type $(1,1,0)$ that intersects
with $X_2$, the induction hypothesis gives
\[H_{X_1}(t_1(X_1)-1,t_2(X_1)-1,k) =
\sum_{m=1}^{k+1}\left(\sum_{n=k+1}^\infty r_n(X_1)\right)\]
and
\[H_{X_2}(t_1(X_2)-1,t_2(X_2)-1,k) =
\sum_{m=1}^{k+1}\left(\sum_{n=k+1}^\infty r_n(X_2)\right)\]
where $t_i(X_j) = |\pi_i(X_j)|$ for $i=1,2$ and $j=1,2$.
Because $t_1 \geq t_1(X_1),t_1(X_2)$ and $t_2 \geq t_2(X_1),t_2(X_2)$,
Lemma \ref{stabilize} then implies that
\begin{eqnarray*}
H_{X_1}(t_1-1,t_2-1,k) &=&
\sum_{m=1}^{k+1}\left(\sum_{n=k+1}^\infty r_n(X_1)\right) ~~\mbox{and} \\
H_{X_2}(t_1-1,t_2-1,k) &=&
\sum_{m=1}^{k+1}\left(\sum_{n=k+1}^\infty r_n(X_2)\right).
\end{eqnarray*}

By Lemma \ref{zerohilbertfunction},
$H_{R/(I(X_1)+I(X_2))}(t_1-1,t_2-1,k) = 0$ for all $k \geq 0$.
Furthermore, $r_n(X) = r_n(X_1)+r_n(X_2)$ for all $n \geq 1$.
Thus, \eqref{ses} gives the desired formula:
\begin{eqnarray*}
H_X(t_1-1,t_2-1,k) & = & \sum_{m=1}^{k+1}\left(\sum_{n=k+1}^\infty r_n(X_1)\right)
+ \sum_{m=1}^{k+1}\left(\sum_{n=k+1}^\infty r_n(X_2)\right) \\
& = & \sum_{m=1}^{k+1}\left(\sum_{n=k+1}^\infty (r_n(X_1)+r_n(X_2))\right)
 =  \sum_{m=1}^{k+1}\left(\sum_{n=k+1}^\infty r_n(X)\right). \\
\end{eqnarray*}
\end{proof}

\begin{example} In this example, we show how
Theorem \ref{sum} can be used to determine values of $H_X$
directly from information about the set of points $X$.

Let $A_i := [1:i] \in \pr^1,$ $B_i := [1:i] \in \pr^1$ and $C_i :=
[1:i] \in \pr^1$, and suppose that $X$ is the following scheme
\[X:=\{P_{111},P_{121},P_{122},P_{132},P_{133},
P_{211}, P_{221},P_{222},P_{232},P_{233}\}\]
where $P_{ijk}=A_i\times B_j\times C_k$.  We can visualize this
set of points as follows:
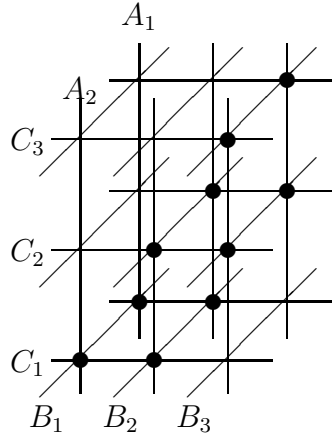
\setlength{\unitlength}{0.49mm}
\begin{center}
\begin{picture}(100,140)

\put(40,55){\line(1,0){60}}
\put(24,39){\line(1,0){60}}
\put(13,36){$C_1$}
\put(13,66){$C_2$}
\put(13,96){$C_3$}
\put(40,85){\line(1,0){60}}
\put(24,69){\line(1,0){60}}
\put(40,115){\line(1,0){60}}
\put(24,99){\line(1,0){60}}

\put(48,45){\line(0,1){80}}
\put(43,130){$A_1$}
\put(32,30){\line(0,1){80}}
\put(27,110){$A_2$}
\put(68,45){\line(0,1){80}}
\put(52,30){\line(0,1){80}}
\put(88,45){\line(0,1){80}}
\put(72,30){\line(0,1){80}}

\put(21,29){\line(1,1){35}}
\put(18,21){$B_1$}
\put(38,21){$B_2$}
\put(58,21){$B_3$}
\put(41,29){\line(1,1){35}}
\put(61,29){\line(1,1){35}}

\put(21,59){\line(1,1){35}}
\put(41,59){\line(1,1){35}}
\put(61,59){\line(1,1){35}}

\put(21,89){\line(1,1){35}}
\put(41,89){\line(1,1){35}}
\put(61,89){\line(1,1){35}}

\multiput(48,55)(20,0){2}{\circle*{4}}
\multiput(68,85)(20,0){2}{\circle*{4}}
\multiput(88,115)(20,0){1}{\circle*{4}}
\multiput(32,39)(20,0){2}{\circle*{4}}
\multiput(52,69)(20,0){2}{\circle*{4}}
\multiput(72,99)(20,0){1}{\circle*{4}}
\end{picture}
\end{center}
Note that in the above picture, the vertical lines represent
the lines of type $(1,1,0)$.

From $X$, we can compute $r_i(X)$ for all $i$.  First, let
$L_{A_i}$ denote the $(1,0,0)$ form that vanishes at all points of
$\prthree$ whose first coordinate is $A_i$, and similarly, let
$L_{B_j}$ denote the $(0,1,0)$ form that vanishes at all points
whose second coordinate is $B_j$. Then the line of type $(1,1,0)$
given by $(L_{A_1},L_{B_1})$ contains one point of $X$, the line
of type $(1,1,0)$ given by $(L_{A_1},L_{B_2})$ contains two points
of $X$, the line of type $(1,1,0)$ given by $(L_{A_1},L_{B_3})$
contains two points,  the line of type $(1,1,0)$ given by
$(L_{A_2},L_{B_1})$ contains one point, the line of type $(1,1,0)$
given by $(L_{A_2},L_{B_2})$ contains two points, and the line of
type $(1,1,0)$ given by $(L_{A_2},L_{B_3})$ contains two points.

From this data, and from $X$, we have $t_1 = 2$ and $t_2 = 3$,
and $r_1(X) = 2$, $r_2(X) = 4$, and $r_i(X) = 0$ for $i \geq 3$.
So,
\begin{eqnarray*}
H_X(1,2,0) &= &r_1(X) + r_2(X) = 6 \\
H_X(1,2,1) &= &r_1(X) + r_2(X) + r_2(X) = 10 \\
H_X(1,2,2) &=& r_1(X) + r_2(X) + r_2(X)  = 10 = H_X(1,2,l)
~~\mbox{for $l \geq 2$.}
\end{eqnarray*}
Note that by Lemma \ref{stabilize}, we have actually computed an
infinite number of values of $H_X$.  For example, for all $(i,j,2)
\succeq (1,2,2)$, we have $H_X(i,j,2) = H_X(1,2,2) =10$.

As a reminder, a similar result to Theorem \ref{sum} also holds
for the number of points on a line of type $(1,0,1)$ and $(0,1,1)$.  Since
$t_3 = 3$, we can also compute the values of $H_X(t_1-1,k,t_3-1)$ and
$H_X(k,t_2-1,t_3-1)$ for all $k \geq 0$.  We omit the details.
\end{example}

\begin{proof}[Proof of Theorem~\textup{\ref{genGMR}.}]
For any $k \geq 0$,
\begin{eqnarray*}
d_{t_1-1,t_2-1,k} & =& H_X(t_1-1,t_2-1,k) - H_X(t_1-1,t_2-1,k-1) \\
& = & \sum_{m=1}^{k+1}\left(\sum_{n=k+1}^\infty r_n(X)\right) -
\sum_{m=1}^{k}\left(\sum_{n=k}^\infty r_n(X)\right) =  \sum_{n=k+1}^\infty r_n(X).
\end{eqnarray*}
Thus
\[d_{t_1-1,t_2-1,k}-d_{t_1-1,t_2-1,k+1} = \sum_{n=k+1}^\infty r_n(X) - \sum_{n=k+2}^\infty
r_n(X) = r_{k+1}(X).\]
\end{proof}

\begin{remark}\label{scale}
We point out that although Theorem \ref{genGMR} is only stated
for points in $\prthree$, one can adapt the proofs to scale
this result to $\mathbb{P}^1 \times \cdots \times \mathbb{P}^1$ ($r$ copies).
In particular, one defines a {\it line of type
$(\underbrace{1,\ldots,1}_r,0)$}
in $\mathbb{P}^1 \times \cdots \times \mathbb{P}^1$ as the
variety defined by linear forms of degree $e_1,\ldots,e_{r-1}$ where
$e_i$ is the standard basis vector of $\mathbb{N}^r$.  If $t_i = |\pi_i(X)|$
is the number of distinct $i$-th coordinates that appear in $X$,
then the number of points on a lines of type
$(1,\ldots,1,0)$ is encoded into the
sequence \[\{H_X(t_1-1,\ldots,t_{r-1}-1,k)\}_{k \in \mathbb{N}}.\]
Our result can also be seen as a generalization of \cite[Theorem 2.12]{GuMaRa}
which determined the geometric information encoded into
the sequence $\{H_X(t_1-1,k)\}_{k \in \mathbb{N}}$ when $X$ is a set
of points in $\popo$ with $t_1 = |X|$.
\end{remark}

\noindent
{\bf Acknowledgments.} The computer algebra system
CoCoA \cite{C} inspired many of the results of this paper.
The second author acknowledges the support of NSERC.  We thank the 
referee for their careful reading of the paper and their invaluable suggestions.


\end{document}